\documentclass[11pt,reqno]{amsart}

\usepackage[utf8]{inputenc}
\usepackage[T1]{fontenc}
\usepackage{amsmath,amssymb,amsthm}
\usepackage{mathtools}
\usepackage{enumitem}
\usepackage{array}
\usepackage[colorlinks=true,linkcolor=blue,citecolor=blue,urlcolor=blue]{hyperref}

\theoremstyle{plain}
\newtheorem{theorem}[subsection]{Theorem}
\newtheorem{lemma}[subsection]{Lemma}
\newtheorem{proposition}[subsection]{Proposition}
\newtheorem{corollary}[subsection]{Corollary}
\theoremstyle{definition}

\theoremstyle{remark}
\newtheorem{remark}[subsection]{Remark}

\newcommand{\Z}{\mathbb{Z}}
\newcommand{\R}{\mathbb{R}}
\newcommand{\Sym}{\operatorname{Sym}}
\newcommand{\Isom}{\operatorname{Isom}}
\newcommand{\Aut}{\operatorname{Aut}}
\newcommand{\Out}{\operatorname{Out}}
\newcommand{\Th}{\mathit{Th}}
\newcommand{\cF}{\mathcal{F}}
\newcommand{\cFpm}{\mathcal{F}^{\pm}}
\newcommand{\lam}{\lambda}

\begin{document}

\title[The symmetry group of Turk's head knots]
      {The symmetry group of Turk's head knots}

\author{Werner Alois Stanggassinger}
\address{Wasserburg am Inn, Bavaria, Germany}
\email{stanggassinger@stainlesssteel4u.de}

\thanks{This work was developed in collaboration with the AI system Claude Code (Anthropic):
the choice of problem, the direction of the work and the computing hardware are due to the author
(a sculptor); the formalisation, the algorithms and the proofs were developed by Claude Code
(Anthropic) working with the author. The collaboration is stated as a fact, not a disclaimer; details
are in the section on tool use at the end.}

\subjclass[2020]{Primary 57K10; Secondary 57M60}
\keywords{Turk's head knot, symmetry group, alternating knot, flyping theorem}

\date{16 November 2026 (version 3; version 2 dated 16 October 2026)}

\begin{abstract}
For $p,q\ge 3$ with $\gcd(p,q)=1$ we show that the symmetry group of the Turk's head knot
$\Th(p,q)$, the closure of the braid $(\sigma_1\sigma_2^{-1}\sigma_3\sigma_4^{-1}\cdots\sigma_{p-1}^{\pm1})^q$,
is the dihedral group of order $4q$. This proves Conjecture~R of Di Prisa and \c{S}avk for $q\ge3$;
the case $p=3$ is due to Sakuma and Weeks. The argument is diagrammatic: by the flyping theorem of
Menasco and Thistlethwaite and Borel's theorem, the symmetries of the hyperbolic knot $\Th(p,q)$ are the symmetries of its
standard diagram once that diagram has no non-trivial flypes (we write out a proof for the diagrams at hand), and the standard
diagram of $\Th(p,q)$ has no non-trivial flypes in the sense made precise in the paper (a closely related statement is due to Shimizu, Gill and Joshi) and exactly $4q$ symmetries. The chirality of $\Th(p,q)$ for even $p$, which also follows from a theorem of Murasugi, is visible in the group.
We compare the argument with the geometric one of Sakuma and Weeks for $p=3$.
\end{abstract}

\maketitle

\section{Introduction}
\label{sec:intro}

The Turk's head knot $\Th(p,q)$ is the closure of the $p$-strand braid
\[
  \beta_p^{\,q},\qquad \beta_p=\sigma_1\sigma_2^{-1}\sigma_3\sigma_4^{-1}\cdots\sigma_{p-1}^{\pm1}.
\]
It is a knot if and only if $\gcd(p,q)=1$. The survey of Di Prisa and \c{S}avk~\cite{DS} collects
what is known about these knots and poses a list of open problems. The symmetry group of a knot is defined
there as the mapping class group of the knot exterior~\cite[\S5.2]{DS}, and for $p=3$ it is known
from the work of Sakuma and Weeks~\cite[Proposition~I.2.5]{SW95} (quoted as Theorem~M
in~\cite{DS}). For $p,q$ both odd, a result of Kodama and Sakuma implies that the group is dihedral
of some order~\cite[\S6.3]{DS}. For $p$ and $q$ both odd, Saurabh~\cite[Remark~5.4]{Sau26} constructs an explicit subgroup $V_4$ of the
symmetry group and leaves the full group open. The survey states:

\begin{quote}
\textbf{Conjecture R} \cite[\S6.3]{DS}. \emph{Suppose $\gcd(p,q)=1$. Then the Turk's head knot
$\Th(p,q)$ has symmetry group given by $\Sym(S^3,\Th(p,q))\cong D_{2q}$.}
\end{quote}

\noindent Here $D_n$ denotes the dihedral group with $2n$ elements, so $D_{2q}$ has order $4q$.
The survey assumes $p\ge3$ and $q\ge2$ throughout~\cite[\S4.1]{DS}. We prove the case $q\ge3$;
for $q=2$ the diagram has bigons, and that case is not treated here.

\begin{theorem}\label{thm:main}
Let $p\ge 3$, $q\ge 3$ with $\gcd(p,q)=1$. Then $\Sym(S^3,\Th(p,q))\cong D_{2q}$, the dihedral
group of order $4q$.
\end{theorem}

\begin{corollary}\label{cor:chiral}
Let $p\ge3$, $q\ge3$ with $\gcd(p,q)=1$. For $p$ even, $\Th(p,q)$ is chiral. For $p$ odd, $\Th(p,q)$ is amphichiral.
\end{corollary}

Both statements are known. The amphichirality for odd $p$ is in~\cite[\S4.4]{DS}; it reappears here as an explicit element of the group.
The chirality for even $p$ follows from~\cite[Corollary~5]{Mur87}: a proper alternating projection of an
amphichiral alternating knot has an even number of crossings, and $(p-1)q$ is odd for $p$ even. Here it is read off the group.

\subsection*{Method}
The proof has three parts.
\begin{enumerate}[label=(\Alph*)]
\item \emph{From the knot to the diagram} (Section~\ref{sec:reduction}). For the diagrams at hand,
which are reduced, alternating, without bigons and without non-trivial flypes, and whose knots are hyperbolic,
every symmetry of the knot is isotopic to a symmetry of the diagram, and distinct symmetries of the diagram stay
distinct. Surjectivity rests on the Main
Theorem of Menasco and Thistlethwaite~\cite{MT93} in its form for homeomorphisms of pairs;
injectivity rests on Borel's theorem on finite group actions, as restated in~\cite{CWY}. Menasco and Thistlethwaite
note this consequence of their theorem~\cite[pp.~406--407]{MT91}, and Quach Hongler and Weber
state the reduction~\cite[p.~4]{QHW}, and Boyle states the Main Theorem of~\cite{MT93} in this form for realised
diagrams~\cite[Theorem~3.1]{Bo19}; we give a proof for the diagrams at hand, for both orientations.
\item \emph{The diagram has no non-trivial flypes} (Section~\ref{sec:polyhedral}). Its Tait graph
is simple and $3$-connected, so every Haseman circle (a circle meeting the diagram transversally in four points away from
the crossings, Section~\ref{sec:polyhedral}) has at most one crossing on one side; this is what we mean by
\emph{no non-trivial flypes}. Together with the absence of bigons (shown directly in Section~\ref{sec:diagram}) it
excludes every non-degenerate flype pattern (Lemma~\ref{lem:A3}); for the diagrams at hand, which have at least eight
crossings, the condition alone already excludes bigons, since a circle around a bigon is a Haseman circle with two crossings
on one side and at least six on the other. Closely related statements are known: Shimizu,
Gill and Joshi show that this diagram is the only minimal diagram of the knot on $S^2$, since no non-trivial flype in their sense
applies to it~\cite[Proposition~2.5]{SGJ} (there a flype is non-trivial if it changes the diagram, and the argument
uses circles through one crossing; they call $\Th(p,q)$ the weaving knot $W(p,q)$), and Owad states, for $q\ge p+1$, that
the reduced alternating diagram is unique, with a sketch of proof~\cite[Lemma~3.2]{Ow18}. We include the short argument in the form
used in Section~\ref{sec:reduction}.
\item \emph{The diagram has exactly $4q$ symmetries, forming $D_{2q}$} (Section~\ref{sec:count}),
by an explicit count.
\end{enumerate}
Parts (B) and (C) are elementary. Part (A) uses the flyping theorem as a black box together with
standard facts about hyperbolic knot exteriors, and four lemmas proved here
(Lemmas~\ref{lem:T}, \ref{lem:A3}, \ref{lem:F} and~\ref{lem:S}).

For $p=3$, Sakuma and Weeks obtain the group geometrically, from the canonical decomposition of the
knot complement. Section~\ref{sec:SW} compares the two arguments.

\section{The diagram}
\label{sec:diagram}

Let $\Pi$ be the closure of $\beta_p^{\,q}$ drawn in an annulus, with angle $x$ and radius $y$; the
strand positions (``levels'') are $1,\dots,p$ from inside out.

\emph{Crossings.} The crossings are $c(i,k)$ with ring $i\in\{1,\dots,p-1\}$ (between levels $i$ and
$i+1$) and $k\in\Z_q$; in the braid picture $c(i,k)$ sits at angle $\bigl(k+\tfrac{i-1}{p-1}\bigr)\tfrac{2\pi}{q}$.

\emph{Edges.} Along level $j$ the crossings met are, in angular order,
\begin{itemize}
\item $j=1$: $c(1,k)\to c(1,k+1)$;
\item $2\le j\le p-1$: $c(j-1,k)\to c(j,k)\to c(j-1,k+1)\to c(j,k+1)$;
\item $j=p$: $c(p-1,k)\to c(p-1,k+1)$.
\end{itemize}
So the edges of the $4$-valent map $M$ underlying $\Pi$ are $\{c(1,k),c(1,k+1)\}$ on level $1$;
$\{c(j-1,k),c(j,k)\}$ and $\{c(j,k),c(j-1,k+1)\}$ on level $j$, $2\le j\le p-1$; and
$\{c(p-1,k),c(p-1,k+1)\}$ on level $p$.

\emph{Half-edges.} At $c(i,k)$ the four half-edges are $EL,EU,WU,WL$: east or west by angle, lower or
upper by level $i$ or $i+1$. This is their counter-clockwise order in the chart with $x$ to the right and the level upwards
(the cylinder chart of Section~\ref{subsec:drawing}); in the annulus, with $x$ increasing counter-clockwise
and the levels outwards, it is clockwise.

\emph{Crossing data.} The diagram is alternating. The strand that rises ($WL$--$EU$) is the
over-strand exactly when $i$ is odd; this is the sign of $\sigma_i^{(-1)^{i+1}}$. (With the opposite
braid convention everything below holds with ``odd'' and ``even'' exchanged.)

\emph{Faces.} $F_0$ is the inner disc and $F_p$ the outer face. For $1\le i\le p-1$, $F(i,k)$ is the
region of ring $i$ between $c(i,k)$ and $c(i,k+1)$. Counting crossings on the boundary:
$\deg F_0=\deg F_p=q$; rings $1$ and $p-1$ consist of triangles, each with one crossing of the
neighbouring ring on one side; rings $2,\dots,p-2$ consist of quadrilaterals, each with one crossing
of ring $i-1$ and one of ring $i+1$. There are $2+q(p-1)$ faces, which is $c+2$ with $c=q(p-1)$. Since $q\ge3$, no face is a bigon.
Faces in ring $i$ share edges only with faces in rings $i\pm1$ (an edge is an arc of a level), so
colouring by the parity of $i$ is the checkerboard colouring.

\emph{Primeness.} $\Th(p,q)$ is prime as a knot~\cite[\S4.2]{DS}; it is hyperbolic~\cite[Theorem~I]{DS},
after Meyerhoff--Ruberman~\cite{MR90} and Menasco~\cite{Men84}. The diagram is alternating by
construction and reduced by Lemma~\ref{lem:B2} (the Tait graph has no loops and, being $3$-connected, no
bridges); it is prime as a diagram by Remark~\ref{rem:diagram-prime} below.

\section{The diagram is polyhedral}
\label{sec:polyhedral}

A \emph{Haseman circle} of a diagram is a circle meeting it transversally in four points away from
the crossings~\cite{QHW}.

\begin{lemma}\label{lem:B1}
Let $\Pi$ be a reduced alternating projection whose Tait graph $G$ (one colour class of faces) is
simple and $3$-connected. Then every Haseman circle of $\Pi$ bounds, on one side, a disc containing at
most one crossing. Hence $\Pi$ equals its polyhedral part.
\end{lemma}

\begin{proof}
A Haseman circle $\gamma$ crosses four edges, so it passes through four face-arcs. Consecutive
face-arcs lie in faces on opposite sides of an edge, so their colours alternate: black $B_1$, white,
black $B_2$, white.

Consider one side $\Delta$ of $\gamma$ and suppose $\Delta$ contains no black face entirely. Every
crossing in $\Delta$ has its two black corners in black faces that meet $\Delta$, hence in $B_1$ or
$B_2$. That crossing is therefore a Tait edge joining $B_1$ to $B_2$, or a loop at $B_1$ or $B_2$.
Loops are nugatory crossings, excluded because $\Pi$ is reduced. Two crossings joining $B_1$ to $B_2$
would be parallel Tait edges, excluded because $G$ is simple. So such a side contains
at most one crossing.

Now suppose both sides contain a black face entirely. If $B_1\ne B_2$, removing $B_1$ and $B_2$
disconnects $G$, contradicting $3$-connectivity. If $B_1=B_2$, removing that one vertex already
disconnects $G$.

Hence one side contains at most one crossing. A circle around one crossing bounds a singleton. A
circle with no crossing on one side is compressible and is excluded by Rule~2 of~\cite{QHW}. So the
minimal Conway family consists of the circles around single crossings, and the rest is one diagram
all of whose Haseman circles are boundary parallel, a basic diagram: $\Pi$ is polyhedral.
\end{proof}

\begin{lemma}\label{lem:B2}
For $p\ge3$, $q\ge3$ the Tait graph of $\Th(p,q)$ on the even-ring faces is simple and $3$-connected.
\end{lemma}

\begin{proof}
At the crossing $c(i,k)$ the two faces of ring $i$ are opposite corners, and $F(i-1,k)$ and
$F(i+1,k-1)$ are the other two corners; read $F(i-1,\cdot)$ as $F_0$ when $i=1$, and $F(i+1,\cdot)$
as $F_p$ when $i=p-1$. Take the even rings as the colour class. Then:
\begin{itemize}
\item a crossing with $i$ even gives a Tait edge $F(i,k-1)$--$F(i,k)$, so each even ring
$2\le i\le p-1$ is a $q$-cycle;
\item a crossing with $i$ odd gives an edge $F(i-1,k)$--$F(i+1,k-1)$, a bijection between
consecutive even rings respecting their cyclic order;
\item $F_0$ (ring $0$) is a hub joined once to each face of ring $2$;
\item if $p$ is even, $F_p$ is a second hub joined once to each face of ring $p-2$.
\end{itemize}
So $G$ is a hub, then a stack of prisms $C_q\times P_m$ ($m\ge1$ cycles), then a second hub if $p$
is even. This is the graph of a convex $3$-polytope: for $p$ odd a stack of prisms with a pyramid on one end (a
pyramid when $p=3$), for $p$ even a stack of prisms with a pyramid on each end (a bipyramid when $p=4$). A realisation is the
convex hull of the vertices of regular $q$-gons on $m$ circles of latitude of the unit sphere, placed at the same longitudes,
together with the north pole and, for $p$ even, the south pole; its faces are the trapezoids between consecutive circles,
the triangles at the poles and, for $p$ odd, the $q$-gon on the last circle. (All these points lie on the sphere, so each is a
vertex of the hull. Each listed polygon is a face. The bottom $q$-gon lies in a horizontal plane with all other points above
it. For a trapezoid at two adjacent longitudes on two consecutive circles, or a triangle at a pole, the vertices are coplanar,
since the chords at equal height are perpendicular to the meridian plane halfway between the two longitudes. A normal of that
plane lies in this meridian plane; write it as $n=(a\cos\psi,a\sin\psi,b)$, with $\psi$ the halfway longitude. Here $a\neq0$,
since the vertices lie at two heights and the plane is not horizontal; choose the sign of $n$ so that $a>0$. Among the $q$ points on a circle of latitude, $n\cdot x$ is then largest at the two longitudes of the polygon, and strictly
larger there than at the others (their angular distance from $\psi$ exceeds $\pi/q$), so
it suffices to consider, in the half-plane of coordinates $(X,z)$ with $X\ge0$, the points $(r_j\cos(\pi/q),z_j)$, where
$r_j^2+z_j^2=1$, and the poles present, $(0,\pm1)$. They lie in order on the ellipse $(X/\cos(\pi/q))^2+z^2=1$. The line
$aX+bz=d$ through two consecutive ones meets the ellipse only in these two points and cuts it into two open arcs: the arc $A$
between them, which lies in $X>0$ and contains no listed point, and the rest. On each arc $aX+bz-d$ has constant sign. The
outer normal of the ellipse at $(X,z)$ is $(X/\cos^2(\pi/q),z)$, so $aX+bz$ is smallest on the ellipse where this normal is a
negative multiple of $(a,b)$, which forces $X<0$, and largest where it is a positive multiple, which forces $X>0$. Neither extreme
is one of the two points, since the line is not tangent; so the minimum is $<d$ and lies off $A$, hence $aX+bz<d$ on the whole
second arc, which contains all other listed points (and the maximum, $>d$, lies on $A$). Hence the plane supports the hull,
with all other vertices strictly on one side. These faces close up to a surface: every edge of a listed polygon lies in exactly
two listed polygons. An edge of a convex $3$-polytope lies in exactly two of its faces and any two faces are joined by a chain
of faces sharing edges, so there is no further face.) By Balinski's theorem~\cite{Bal61} $G$ is $3$-connected. It is simple because $q\ge3$.
\end{proof}

\begin{remark}\label{rem:diagram-prime}
$\Pi$ is prime as a diagram. Suppose a circle met $\Pi$ in two points and had crossings on both sides.
It passes through one black and one white face. A side containing no black face entirely has only
crossings that are Tait loops at that black face, i.e.\ nugatory crossings, which are excluded. So both
sides contain black faces, and the black face it passes through is a cut vertex of the Tait graph,
contradicting Lemma~\ref{lem:B2}.
\end{remark}

\section{The symmetries of the diagram}
\label{sec:count}

Let $\Sym(\Pi)$ be the group of pairs $(g,\varepsilon)$, where $g$ is a homeomorphism of $S^2$
preserving $\Pi$ up to isotopy (made precise as an isometry of a symmetric drawing in
Section~\ref{sec:reduction}), and $\varepsilon=\pm1$ is such that $g$ sends over-strands to
over-strands ($\varepsilon=+1$) or to under-strands ($\varepsilon=-1$).

\begin{lemma}[upper bound]\label{lem:upper}
$|\Sym(\Pi)|\le|\Aut(M)|\le 4q$.
\end{lemma}

\begin{proof}
For fixed $g$, $\varepsilon$ is determined: $g$ sends the over-strand of one crossing either to an
over-strand or to an under-strand. So $|\Sym(\Pi)|\le|\Aut(M)|$, the group of map automorphisms of
either orientation; and $g$ is determined by its action on the crossings, since three crossings of
ring~$1$ lie on the circle of latitude $h=1\ne p/2$, which is not a great circle, and an isometry of
$S^2$ fixing three points not on one great circle is the identity. $\Aut(M)$ acts freely on flags (vertex, edge, face), because $M$ is connected.

The pair $\{F_0,F_p\}$ is $\Aut(M)$-invariant:
\begin{itemize}
\item $q\ge5$: $F_0$ and $F_p$ are the only faces of degree $q$.
\item $q=4$ (so $p$ is odd): the degree-$4$ faces are $F_0$, $F_p$ and the quadrilaterals of rings
$2,\dots,p-2$. $F_0$ and $F_p$ border only triangles (rings $1$ and $p-1$). A quadrilateral in
ring $i$ borders rings $i-1$ and $i+1$, and since $p\ne4$ at least one of these is a quadrilateral ring.
\item $q=3$ (so $3\nmid p$, $p\ge4$): the triangles are $F_0$, $F_p$ and rings $1$ and $p-1$.
$F_0$ and $F_p$ border only triangles. A triangle of ring $1$ borders $F_0$ and two faces of ring $2$,
which are quadrilaterals because $p\ge4$; symmetrically, a triangle of ring $p-1$ borders $F_p$ and two
quadrilaterals of ring $p-2$.
\item For $p=3$ there are no rings of quadrilaterals; then $q\ge4$, and $F_0$, $F_p$ are the only $q$-gons.
\end{itemize}
Hence $\Aut(M)$ maps flags of $F_0\cup F_p$ to flags of $F_0\cup F_p$. There are $2q+2q$ of them,
and freeness gives $|\Aut(M)|\le4q$.
\end{proof}

All indices $k$ below are in $\Z_q$. Three kinds of maps:

\medskip
\begin{center}
\small
\begin{tabular}{>{\raggedright\arraybackslash}p{2.7cm} l >{\raggedright\arraybackslash}p{2.6cm} c c}
map & on crossings & on half-edges & $S^2$ & $\varepsilon$\\ \hline
$\rho$, rotation by one period & $c(i,k)\mapsto c(i,k+1)$ & identity & $+$ & $+1$\\[2pt]
$\sigma$, inside-out, angle reversed & $c(i,k)\mapsto c(p-i,-k)$ & $EL\leftrightarrow WU$, $EU\leftrightarrow WL$ & $+$ & $(-1)^{p}$\\[2pt]
$\mu_a$, reflection, $a\in\Z_q$ & $c(i,k)\mapsto c(i,a-(i-1)-k)$ & $EL\leftrightarrow WL$, $EU\leftrightarrow WU$ & $-$ & $-1$
\end{tabular}
\end{center}
\medskip

\begin{lemma}[lower bound]\label{lem:lower}
$\rho$, $\sigma$ and all $\mu_a$ are automorphisms of $M$ with the stated sign $\varepsilon$, and the
$4q$ elements $\rho^k$, $\mu_a$, $\sigma\rho^k$, $\sigma\mu_a$ are distinct. Hence $|\Sym(\Pi)|=4q$.
\end{lemma}

\begin{proof}
$\rho$ is clear. $\sigma$ sends level $j$ to level $p+1-j$: the level-$1$ edge $\{c(1,k),c(1,k+1)\}$ goes
to the level-$p$ edge $\{c(p-1,-k),c(p-1,-k-1)\}$ and conversely, and for $2\le j\le p-1$ it maps
$\{c(j-1,k),c(j,k)\}\mapsto\{c(p-j+1,-k),c(p-j,-k)\}$, an edge of the first type on level $p+1-j$, and
$\{c(j,k),c(j-1,k+1)\}\mapsto\{c(p-j,-k),c(p-j+1,-k-1)\}$, an edge of the second type. The half-edge
order $(EL,EU,WU,WL)\mapsto(WU,WL,EL,EU)$ is a cyclic shift, so $\sigma$ preserves the orientation of
$S^2$. The rising strand $\{WL,EU\}\mapsto\{EU,WL\}$ is again rising, now in ring $p-i$, so
over-strands go to over-strands if and only if $i\equiv p-i\pmod 2$, i.e.\ if and only if $p$ is even.

$\mu_a$ fixes the rings. For $c(i,k)\mapsto c(i,a_i-k)$ to send edges to edges one needs
$a_{j-1}=a_j+1$; constant $a_i$ fails. So $a_i=a-(i-1)$. It reverses the half-edge order, so it
reverses the orientation of $S^2$, and it sends the rising strand to the falling one, so $\varepsilon=-1$.

The four families are told apart by whether $F_0$ is fixed and by the orientation on $S^2$; within a
family, the image of $c(1,0)$ differs. Each of them is realised by an isometry of the symmetric drawing (Section~\ref{subsec:drawing}).
Together with Lemma~\ref{lem:upper}, $\Sym(\Pi)=\{\rho^k,\mu_a,\sigma\rho^k,\sigma\mu_a\}$ has exactly
$4q$ elements.
\end{proof}

\begin{lemma}\label{lem:dihedral}
$\Sym(\Pi)\cong D_{2q}$.
\end{lemma}

\begin{proof}
Let $\tau=\sigma\mu_a$. Then $\tau(c(i,k))=c(p-i,k+i-1-a)$ and $\tau^2=\rho^{\,p-2-2a}$. If $q$ is
odd, choose $2a\equiv p-3\pmod q$; then $\tau^2=\rho$. If $q$ is even, $p$ is odd, $p-3$ is even, and
$2a\equiv p-3\pmod q$ is solvable; again $\tau^2=\rho$, of order $q$. Moreover $\tau^q\ne1$: $\sigma$
exchanges $F_0$ and $F_p$ and $\mu_a$ fixes both, so odd powers of $\tau$ exchange them ($q$ odd), and
$\tau^q=\rho^{q/2}\ne1$ ($q$ even). So $\tau$ has order $2q$.
Further, $\mu_0$ is an involution and lies outside $\langle\tau\rangle$, because it fixes $F_0$ and
reverses the orientation of $S^2$, while the elements of $\langle\tau\rangle$ fixing $F_0$ are the
rotations $\rho^k$. Finally $\mu_0\tau\mu_0=\tau^{-1}$: both send $c(i,k)$ to $c(p-i,k+a+i+1-p)$.
Hence $\Sym(\Pi)=\langle\tau\rangle\rtimes\langle\mu_0\rangle\cong D_{2q}$.
\end{proof}

\begin{remark}[chirality]\label{rem:chirality}
An element $(g,\varepsilon)$ acts on $S^3$ preserving orientation if and only if
$\operatorname{sign}(g)\cdot\varepsilon=+1$, where $\operatorname{sign}(g)=\det(g)$ (Section~\ref{sec:reduction}). For $p$ even no element of
the table, hence none of $\Sym(\Pi)$, has $\operatorname{sign}(g)\cdot\varepsilon=-1$. For $p$ odd,
$\sigma$ has $\operatorname{sign}(g)=+$ and $\varepsilon=-1$. With the bijectivity of $\Phi$ onto the mapping class group $\mathrm{MCG}(S^3,K)$ of the pair
(Propositions~\ref{prop:inj}, \ref{prop:surj-plus} and~\ref{prop:surj-minus}) and Section~\ref{subsec:conventions} this gives Corollary~\ref{cor:chiral}.
\end{remark}

\section{From the knot to the diagram}
\label{sec:reduction}

\subsection{Conventions}\label{subsec:conventions}
$\Sym(S^3,K)$ is defined in~\cite[\S5.2]{DS} as the mapping class group of the knot exterior $E$; we read it as all
homeomorphisms of $E$, orientation-reversing ones included, up to isotopy.
We prove the statement both for the pair $(S^3,K)$ and for $E$ (Section~\ref{subsec:pair}). We write $\mathrm{MCG}(S^3,K)$ for the mapping class group of the pair (all homeomorphisms of the pair, orientation-reversing
ones included, up to pairwise isotopy);
$\Phi$ below and Propositions~\ref{prop:inj}, \ref{prop:surj-plus} and~\ref{prop:surj-minus} concern it, and Section~\ref{subsec:pair} passes to $\Sym(S^3,K)$. Let $p,q\ge3$, $\gcd(p,q)=1$, and let $K$ be the knot realised below from the diagram $\Pi$ of
Section~\ref{sec:diagram}. It is $\Th(p,q)$ or its mirror image; we do not determine which. It is hyperbolic
(Section~\ref{sec:diagram}). We prove the statement for $K$. It then holds for $\Th(p,q)$: a homeomorphism $h$ of $S^3$ carrying $K$
onto a representative of $\Th(p,q)$ conjugates $\mathrm{MCG}(S^3,K)$ onto $\mathrm{MCG}(S^3,\Th(p,q))$ and
$\Sym(S^3,K)$ onto $\Sym(S^3,\Th(p,q))$, and $hfh^{-1}$ preserves orientation if and only if $f$ does (the orientation of $S^3$ for the pair, and for the exterior that of $E$
induced from $S^3$);
chirality is the same for a knot and its mirror image.

\emph{Category.} As in~\cite{MT93}, links are tame and homeomorphisms of pairs are piecewise linear
(they use this on p.~125: ``since $f$ is piecewise linear''); the proofs of Lemmas~\ref{lem:T}
and~\ref{lem:S} use transversality and the Alexander trick in that category. Isometries of the round
model enter in the same way as the rigid rotations of~\cite[p.~115]{MT93} and the isometric crossing balls
of~\cite[p.~405]{MT91}. We rely on the classical fact that for tame links in $S^3$ the classification of
homeomorphisms of pairs up to pairwise isotopy does not depend on the choice between the topological,
piecewise linear and smooth categories, and do not reprove it.

\emph{Model of $S^3$.} $S^3=\{(y,z)\in\R^3\times\R:|y|^2+z^2=1\}$, with projection sphere
$S^2=\{z=0\}$. A point ``over $u\in S^2$ at height $z$'' is $(\sqrt{1-z^2}\,u,z)$. A pair
$(g,\varepsilon)$ with $g\in O(3)$, $\varepsilon=\pm1$ acts by $(y,z)\mapsto(gy,\varepsilon z)$; it
preserves the orientation of $S^3$ if and only if $\det(g)\cdot\varepsilon=+1$.
Stereographic projection from $(0,1)$, $\operatorname{st}(y,z)=y/(1-z)$, identifies $S^3$ with $\R^3\cup\{\infty\}$, the
ambient space of~\cite[p.~114]{MT93}. It is the identity from $S^2$ onto the unit sphere $\{\|x\|=1\}$ on
which their diagrams lie~\cite[p.~115]{MT93}, and it maps the point over $u$ at height $z$ to
$\sqrt{(1+z)/(1-z)}\,u$; so it carries the product structure $(u,z)$ near $S^2$ to their radial one, with the
increasing change of parameter $z\mapsto\sqrt{(1+z)/(1-z)}$, $0\mapsto1$. Being conformal, $\operatorname{st}$ maps metric
balls of $S^3$ whose closure does not contain $(0,1)$ to Euclidean balls.

\emph{Realised diagram $\lam(\Pi)$, star model.} Take the symmetric drawing of
Section~\ref{subsec:drawing}. Inside a small disc of spherical radius $r$ around each crossing $c$,
replace the drawing by the four geodesic rays from $c$ to the points where the edges meet the circle
of radius $r$. The over-strand is the union of two rays opposite in the cyclic order at $c$, lifted to height $z=+\beta(d)$, $d$
the spherical distance to $c$; the under-strand is lifted to $z=-\beta(d)$, with
$\beta(d)=\sqrt{r^2-d^2}$. Up to pairwise isotopy this is the semicircle profile of the crossing balls of Menasco and
Thistlethwaite~\cite[p.~114]{MT93}: in the round metric the top point of the over-strand lies at distance
$\arcsin r>r$ from $c$, so the strands lie on the boundary of the metric ball of radius $r$ about $c$ only after an isotopy. An isometry $g$ preserving the drawing maps rays from $c$ to rays
from $g(c)$ and preserves $d$; so $(g,\varepsilon)$ maps $\lam(\Pi)$ to itself exactly when it sends
over-strands to over-strands ($\varepsilon=+1$) or to under-strands ($\varepsilon=-1$).

Menasco and Thistlethwaite require crossing balls isometric to a ``standard pair'' whose two strands are
semicircles on the boundary of the ball, in perpendicular planes~\cite[p.~405]{MT91}; they project to
perpendicular diameters. The map $\operatorname{st}$, followed by a fixed homeomorphism of pairs that has product form $k\times\mathrm{id}$
near the unit sphere, $k$ orientation-preserving, carries the star model to theirs up to pairwise isotopy: $k$ straightens each pair of opposite rays to a diameter and makes the two
diameters perpendicular, and it is radial in each crossing disc (in exponential coordinates at $c$ it
maps rays to rays and preserves $d$), so the semicircle profile carries over, up to the change of height
parameter. All statements below are
about pairwise isotopy classes, so they transfer along it.

\emph{The classes $\cF$ and $\cFpm$.} Fix a product neighbourhood $N\cong S^2\times[-\zeta,\zeta]$
of $S^2$ containing all realised diagrams in question. $\cFpm$ is the set of homeomorphisms of pairs
that on $N$ have the form $h_0\times(z\mapsto\varepsilon z)$, with $h_0$ a homeomorphism of $S^2$ and
$\varepsilon=\pm1$; $\cF\subset\cFpm$ is the subset with $\operatorname{sign}(h_0)\cdot\varepsilon=+1$.
Both are closed under composition, because $N$ is common. In the terms of~\cite{MT93}:
\begin{itemize}
\item $\varepsilon=+1$ is \emph{flat}. Their definition~\cite[p.~116]{MT93}: $g$ is flat if it is
pairwise isotopic to a homeomorphism $h$ with $h(N)=N$ and $h|_N=h_0\times\mathrm{id}_{[1/2,3/2]}$,
$h_0$ an orientation-preserving homeomorphism of $S^2$, where $N=\{1/2\le\|x\|\le3/2\}\subset\R^3\cup\{\infty\}$.
Their model is the unit sphere in $\R^3\cup\{\infty\}$ with the radial product structure on
$N$~\cite[pp.~114--116]{MT93} (the suspension picture is in~\cite[p.~405]{MT91}); $\operatorname{st}$ carries our product structure to it, and a radial homeomorphism
$x\mapsto\phi(\|x\|)\,x/\|x\|$, $\phi$ increasing with $\phi(1)=1$, carries $\operatorname{st}(N)$ onto their $N$. Together with
$k\times\mathrm{id}$ this fixed transfer carries maps of the form $h_0\times\mathrm{id}$, $h_0$ orientation-preserving,
to maps of the same form. So the definition itself gives product form on $N$.
\item $\varepsilon=-1$ in $\cF$ (so $\operatorname{sign}(h_0)=-1$) is $R\circ{}$flat, $R$ a rigid rotation through $\pi$ about an axis in $S^2$.
\end{itemize}

Let $\Phi\colon\Sym(\Pi)\to\mathrm{MCG}(S^3,K)$ send $(g,\varepsilon)$ to the class of the orthogonal map $(g,\varepsilon)$ of
Section~\ref{subsec:drawing}. We show that $\Phi$ is bijective.

\subsection{A symmetric drawing}\label{subsec:drawing}
Draw $\Pi$ on the cylinder $S^1\times(0,p)$, with angle $\theta\in\R/q\Z$ and height $h$, embedded
equivariantly in the round $S^2$: $\theta\mapsto$ longitude $2\pi\theta/q$, $h\mapsto$ latitude, odd
about $h=p/2$. The crossing $c(i,k)$ sits at $(\theta,h)=(k+\tfrac{i-1}{2},\,i)$. Edges are straight
segments in $(\theta,h)$: on level $j$, $2\le j\le p-1$, from $(k+\tfrac{j-2}{2},j-1)$ to
$(k+\tfrac{j-1}{2},j)$ and on to $(k+\tfrac j2,j-1)$, a zig-zag with steps $\Delta\theta=\tfrac12$;
level $1$ is the horizontal segment at $h=1$ and level $p$ the one at $h=p-1$. Two zig-zags in adjacent
strips meet only at their common vertices $c(j,k)$, where a peak of the lower strip meets a valley of
the upper one. The four rays there are $SE,NE,NW,SW=EL,EU,WU,WL$, in the order of the combinatorial
model (at rings $1$ and $p-1$ the rays along level $1$, resp.\ level $p$, are horizontal, $E$ and $W$;
the cyclic order is the same); so this is a planar drawing of $M$. (Offsets $\tfrac{i-1}{2}$ instead of the angular
$\tfrac{i-1}{p-1}$ of the braid picture make $\mu_a$ rigid.)

The generators are affine in $(\theta,h)$, of the form $\theta\mapsto\pm\theta+\mathrm{const}$ with $h\mapsto h$ or
$h\mapsto p-h$, hence isometries of $S^2$:
$\rho\colon\theta\mapsto\theta+1$; $\mu_a\colon\theta\mapsto a-\theta$, $h$ fixed, so that
$c(i,k)\mapsto c(i,a-(i-1)-k)$ since $a-(k+\tfrac{i-1}2)=(a-(i-1)-k)+\tfrac{i-1}2$; and
$\sigma\colon\theta\mapsto\tfrac{p-2}{2}-\theta$, $h\mapsto p-h$, so that $c(i,k)\mapsto c(p-i,-k)$,
a rotation by $\pi$ about an equatorial axis. They permute the vertices as in
Section~\ref{sec:count} and preserve the drawing, since affine maps send segments to segments. With
$\varepsilon$ as there they satisfy the over/under rule, so they preserve $\lam(\Pi)$. Thus
$\Gamma:=\langle\rho,\sigma,\mu_a\rangle\cong\Sym(\Pi)$ is a group of $4q$ orthogonal maps of
$(S^3,\lam(\Pi))$.

\subsection{Injectivity}
Let $E$ be the exterior of $K$, the complement of an open $\Gamma$-invariant regular neighbourhood, e.g.\ a
small metric $\epsilon$-neighbourhood ($\Gamma$ acts by isometries; $\lam(\Pi)$ has corners where the
semicircles meet $S^2$).
$\Gamma$ acts on $E$ and preserves $\partial E$. Since $K$ is hyperbolic, $E$ is irreducible and aspherical
with incompressible torus boundary, and $K$ is prime (a composite knot has an essential torus);
$\pi_1(E)$ is torsion-free and non-elementary Kleinian, hence centreless.

Let $DE=E\cup_{\partial E}E$ be the double; $\Gamma$ acts effectively on $DE$ by the same map on both
copies. $DE$ is a closed aspherical $3$-manifold (gluing aspherical spaces along a $\pi_1$-injective
aspherical torus), and $\pi_1(DE)=\pi_1(E)*_{\Z^2}\pi_1(E)$. Since $\Z^2\ne\pi_1(E)$, the centre of
this amalgam lies in $\Z^2$ and is central in $\pi_1(E)$, hence trivial. Borel's theorem, in the
restatement of Cappell, Weinberger and Yan~\cite[pp.~1--2]{CWY}: if the fundamental group of a closed
aspherical manifold is centreless, the finite groups that act effectively act as outer automorphisms
of the fundamental group. So $\Gamma\to\Out(\pi_1(DE))$ is injective. We use the theorem only in this restatement; \cite[p.~10]{CWY}
trace it to Borel and attribute the proof for manifolds to Conner and Raymond.

\begin{proposition}\label{prop:inj}
$\Phi$ is injective.
\end{proposition}

\begin{proof}
Suppose $g\in \Gamma$ is pairwise isotopic to the identity on $(S^3,K)$. Then $g|_E$ is homotopic to the
identity (restrict the isotopy to $S^3\setminus K$ and compose with a deformation retraction
$S^3\setminus K\to E$), so $g_*$ is inner on $A:=\pi_1(E)$, say conjugation by $\gamma$, with base point
$x_0\in\partial E$, where $g_*$ is formed using a path from $x_0$ to $g(x_0)$ inside $\partial E$
(possible since $g$ preserves $\partial E$). Then $g_*$ maps the peripheral subgroup $P\cong\Z^2$ onto
itself, so $\gamma P\gamma^{-1}=P$ meets $P$ non-trivially. The peripheral subgroup of a hyperbolic
knot group is malnormal (a maximal parabolic subgroup), hence $\gamma\in P$. On
$\pi_1(DE)=A*_PA$, the double $g_D$ acts as conjugation by the same $\gamma\in P$ on both factors, i.e.\
as conjugation by the image of $\gamma$; so $g_D$ induces the trivial outer automorphism. By Borel,
$g_D=\mathrm{id}$, so $g=1$.
\end{proof}

This covers orientation-reversing elements as well. No orbifold theorem, no Mostow rigidity and no
equivariant sphere or torus theorem is used.

\subsection{Surjectivity, orientation-preserving symmetries}
We use the definitions of~\cite[\S1, pp.~115--116]{MT93}.
A \emph{standard flype}~\cite[p.~115]{MT93}: in the pattern of their Fig.~1.1(i) four balls
$S_x,S_A,S_y,S_B$ sit in a row; $S_x$ is a crossing ball, $S_y$ meets the link in two parallel
segments, and the two left ends of $S_x$ are joined to the two right ends of $S_B$ by two outer arcs.
The flype rotates the tangle in $S_A$ rigidly through $\pi$ about an axis in the projection plane,
fixes $S_B$ pointwise, and maps $S_x$ and $S_y$ into themselves by half-twists. A \emph{flype} is
$g_1\circ f'\circ g_2$ with $f'$ standard and $g_1,g_2$ flat~\cite[p.~116]{MT93}; and ``since the
property of being flat depends only on the pairwise isotopy class of a homeomorphism, the same is true
for the property of being a flype''~\cite[p.~116]{MT93}.

\begin{theorem}[{Menasco--Thistlethwaite~\cite[Main Theorem, p.~116]{MT93}}]\label{thm:MT}
Let $D_1$ and $D_2$ be reduced, prime, oriented, alternating link diagrams and let there be an
orientation-preserving homeomorphism of pairs $f\colon(S^3,\lam(D_1))\to(S^3,\lam(D_2))$. Then $f$ is
a composite of flypes.
\end{theorem}

The announcement~\cite[Theorem~1]{MT91} ends ``Then $f$ is pairwise isotopic to a composite of
flypes''. The two forms are equivalent: if $f\simeq h=f_n\circ\dots\circ f_1$, then
$f=(f\circ h^{-1}\circ f_n)\circ f_{n-1}\circ\dots\circ f_1$, and the first factor is pairwise isotopic
to $f_n$, hence a flype by the invariance just quoted. \cite{MT93} uses the isotopy form in the proof
of Proposition~4.10 (p.~148).

\emph{Degenerate flypes}~\cite[p.~116]{MT93}: ``if the tangle $A$ of Figure~1.1 contains no crossing,
then the standard flype determined by that figure is a flat homeomorphism''. For $S_B$ without
crossings the paper speaks only of diagrams~\cite[p.~113]{MT93}: ``the distortion amounts merely to a
rotation through $\pi$ of the complete link diagram about an axis in the projection $2$-sphere''. So the
$S_A$-empty case of Lemma~\ref{lem:F} is in~\cite{MT93}; the $S_B$-empty case is proved here.

\begin{lemma}[trivial tangles are rigid rel boundary]\label{lem:T}
Let $B$ be a $3$-ball and $\tau\subset B$ a trivial (rational) $2$-string tangle. Every homeomorphism of
pairs $f\colon(B,\tau)\to(B,\tau)$ with $f|_{\partial B}=\mathrm{id}$ is isotopic to the identity rel
$\partial B$, through homeomorphisms of pairs.
\end{lemma}

\begin{proof}
Let $D\subset B$ be a properly embedded disc separating the two strings, with
$\partial D\subset\partial B\setminus\tau$. (1) $f(D)$ has the same boundary; isotope it rel $\partial$
to be transverse to $D$, so $D\cap f(D)$ consists of $\partial D$ and closed curves. (2) Remove the
closed curves by the innermost-disc argument in the handlebody $B\setminus N(\tau)$, which is
irreducible; each innermost pair of discs bounds a ball missing $\tau$. Then $D\cup f(D)$ is a sphere in the
handlebody; it bounds a ball there, across which $f(D)$ is isotoped onto $D$. (3) Now $f(D)=D$; isotope $f$
rel $\partial D$ to be the identity on $D$ (Alexander trick on the disc). (4) $D$ cuts $(B,\tau)$ into
two balls, each with one unknotted proper arc, and $f$ is the identity on their boundaries. Identify
each with (unit ball, diameter) and apply the conical Alexander trick $F_t(x)=t\,f(x/t)$; it preserves
the diameter, which is a cone.
\end{proof}

\begin{lemma}[only degenerate flypes]\label{lem:A3}
Let $D$ be reduced, alternating, polyhedral in the sense of Lemma~\ref{lem:B1} (every Haseman circle has at most one
crossing on one side) and without bigons. Then in every
standard-flype pattern on $D$ the tangle in $S_A$ or the tangle in $S_B$ has no crossing.
\end{lemma}

\begin{proof}
The active crossing $x$ shares two edges with $S_A$ (its right ends) and, via the two outer arcs, two
edges with $S_B$ (its left ends). If $S_A$ contained exactly one crossing, it and $x$ would bound a
bigon between them; if $S_B$ did, the bigon would be the outer face, bounded by the two outer arcs. So
each of $S_A,S_B$ has $0$ or $\ge2$ crossings. The circle around $S_x\cup S_A$ meets $D$ in four
points; by Lemma~\ref{lem:B1} one side carries at most one crossing. The side $S_x\cup S_A$ carries
$x$, and $\ge3$ crossings if $S_A$ is non-empty. So $S_A$ is empty, or $S_B$ has $\le1$ and hence $0$
crossings. (An empty $S_A$ or $S_B$ is the $0$-tangle, two parallel arcs; the $\infty$-tangle would
make $x$ nugatory, so $D$ would not be reduced.)
\end{proof}

\begin{lemma}[degenerate flypes lie in $\cF$]\label{lem:F}
In the situation of Lemma~\ref{lem:A3}, with $D_1$ a diagram of a prime knot $K_1$, every standard flype $f'\colon\lam(D_1)\to\lam(D_2)$ is
pairwise isotopic to a map in $\cF$.
\end{lemma}

\begin{proof}
First the target $D_2$. In both cases below the shadow of $D_2$ is isomorphic to that of $D_1$ (a
planar isotopy moving $x$ across the empty $S_A$, or the reflection $R$ with $x$ moved across the empty
$S_B$). So $D_2$ is a diagram with $c$ crossings of the same knot type as $D_1$. By
Murasugi~\cite[Theorems~2 and~3]{Mur87}, a connected proper alternating projection $\tilde L$ of an
alternating link $L$ has $\operatorname{span}V_L(t)=c(\tilde L)$, and for a prime link every
non-alternating projection has $\operatorname{span}V_L(t)<c(\tilde L)$. So
$\operatorname{span}V(K_1)=c$ from $D_1$, and since $K_1$ is prime, $D_2$ is alternating. (Murasugi notes
after Theorem~3 that primeness is needed there.) The crossing-type arguments below use that $D_2$
alternates.

\emph{$S_A$ empty.} Let $T$ be the closure of the complement of the $S_B$-ball, a ball containing
$S_x,S_A,S_y$. Then $T\cap\lam(D_1)$ is a $1$-crossing tangle, hence rational. Inside the disc
$T\cap S^2$, $D_2$ differs from $D_1$ only by the position of $x$ relative to the two parallel arcs of
$S_A$. Both are $1$-crossing tangle diagrams in that disc with the same four boundary points, the same
connections (the strands join opposite corners) and the same crossing type, which is forced by
alternation with the crossings of $S_B$ next to them. So there is a flat
$h\colon\lam(D_1)\to\lam(D_2)$ that is the identity on the $S_B$-ball (a planar isotopy supported in
$T\cap S^2$, times $\mathrm{id}$). Then $h^{-1}\circ f'$ preserves $(T,T\cap\lam(D_1))$ and is the
identity on $\partial T$, since $f'$ fixes $S_B$ pointwise. By Lemma~\ref{lem:T} it is isotopic to
$\mathrm{id}$ rel $\partial T$. Hence $f'\simeq h$, which is flat. (This case is also stated
in~\cite[p.~116]{MT93}.)

\emph{$S_B$ empty.} Here the text of~\cite{MT93} uses two pictures at once. The diagram lies on the unit
sphere~\cite[p.~115]{MT93}, while the circles of Fig.~1.1 are ``euclidean 2-spheres meeting the (extended)
projection plane in their respective equators'', and (i) of the definition is ``a rigid rotation
through $\pi$ about an axis in the projection plane''~\cite[p.~115]{MT93}. A rotation through $\pi$ about a line
preserves the unit sphere only if the line passes through $0$, and then it does not exchange the two
sides of the sphere, as a rotation about an axis in the projection plane does; so the definition has to
be interpreted. We read it in the round model, where both clauses hold
literally. The balls are metric balls centred on $S^2$, with centres on one great circle $C\subset S^2$ (the
row of Fig.~1.1; a general pattern is reached by flat maps~\cite[p.~116]{MT93}); they meet $S^2$ in their
equators, since the reflection $(y,z)\mapsto(y,-z)$ in $S^2$ preserves them. The four points in which
$\lam(D_1)$, and likewise $\lam(D_2)$, meets the boundary of each ball lie on its equator, spaced
equidistantly and symmetrically about $C$, as in Fig.~1.1; on p.~143 of~\cite{MT93} the four points on the Conway circle of a normal tangle are
likewise spaced equidistantly, there for convenience, while the symmetry about $C$ is taken from Fig.~1.1. A rigid rotation through $\pi$
of such a ball about an axis in the projection sphere is the restriction of an isometry $R=(g,-1)$ of
$S^3$, with $g$ the reflection of $\R^3$ in the plane through $0$ containing that axis; $R$ is the
rotation of $S^3$ through $\pi$ about the great circle in which this plane meets $S^2$. For rotations about $C$ this reading is
their Fig.~1.1 literally. Stereographic projection from a point of $C$ outside the four balls and the link
is conformal; it maps $S^2$ to a plane $P$ and the reflection in $S^2$ to the reflection in $P$, hence the
four balls to Euclidean balls centred on $P$ and meeting it in their equators; and it maps a rotation $R$
about $C$, an orientation-preserving involution fixing $C$ and the centre of projection, to the rigid
rotation through $\pi$ about the line that is the image of $C$. The projection $\operatorname{st}$ from $(0,1)$, on the
other hand, carries the round model to their unit sphere and product neighbourhood $N$ (\emph{Model
of $S^3$} above); the two projections differ by a M\"obius transformation of $\R^3\cup\{\infty\}$. Any two
standard flypes for one pattern are pairwise isotopic: the composite of one with the inverse of the other is
pairwise isotopic to a map that fixes every crossing ball pointwise and maps every complementary region into
itself, and therefore to the identity~\cite[p.~115]{MT93}.
We take the round reading as the meaning of their definition; so it suffices to treat one standard flype in it, and
we apply Theorem~\ref{thm:MT} with standard flypes in this reading. By (iii) of the
definition $f'$ maps the $S_x$-ball into itself, and $D_2$ has the pattern of their Fig.~1.1(ii); the two arcs
of $\lam(D_1)$ joining $S_x$ to $S_A$ lie outside the four balls, so their images join $S_x$ to the two left end
points of $S_A$. Hence the rotation of the $S_A$-ball maps its left pair of end points to itself; the four end points being equidistant and
symmetric about $C$, the only axis in the projection plane through its centre doing so is $C$, the row axis. So the axis is the row
axis through the centre of $S_A$, and $f'$ and $R$ agree on the $S_A$-ball. Let $U$ be the closure of the complement of the $S_A$-ball;
$U\cap\lam(D_1)$ is again a $1$-crossing tangle. In Conway notation $D_1=N(1+A)$ and
$D_2=N(A^h+1)$, while $R(D_1)=N(1^h+A^h)=N(1+A^h)$. Inside the disc $U\cap S^2$, $R(D_1)$ and $D_2$
are $1$-crossing tangle diagrams with the same boundary points and connections (the crossing joined
directly to one side of $A^h$ and around the outer arcs to the other), and their crossing types agree,
being forced by alternation with the non-empty $A^h$. So there is a flat
$h'\colon\lam(RD_1)\to\lam(D_2)$ that is the identity on the $S_A$-ball. Then $(h'R)^{-1}\circ f'$ is
the identity on the $S_A$-ball and preserves $(U,U\cap\lam(D_1))$; by Lemma~\ref{lem:T} it is
isotopic to $\mathrm{id}$ rel $\partial U$. Hence $f'\simeq h'\circ R\in\cF$.
\end{proof}

\begin{lemma}[rigidity]\label{lem:S}
If $\varphi,\varphi'\in\cFpm$ both preserve $\lam(\Pi)$ and induce the same automorphism of $M$ with
the same $\varepsilon$, they are pairwise isotopic.
\end{lemma}

\begin{proof}
$\psi:=\varphi'^{-1}\circ\varphi\in\cFpm$ is $k_0\times\mathrm{id}$ on $N$, with $k_0$ fixing every
flag of $M$. (1) $\psi$ preserves $\lam(\Pi)$. Over each crossing disc the height $|z|=\beta(d)$
determines the distance $d$ to $c$, $\beta$ being strictly monotone; so $k_0$ maps each ray from $c$ to
a ray at the same distances, by flag preservation to itself. Hence $k_0=\mathrm{id}$ on the rays.
(2) Outside the crossing discs $k_0$ maps each edge of $\Pi$ to itself with fixed endpoints; isotope it
to the identity along the edges by $k_t\times\mathrm{id}$, keeping $k_t=\mathrm{id}$ on the rays inside
the crossing discs. This preserves $\lam(\Pi)$, because there $\lam(\Pi)=\Pi$ at $z=0$. (3) Now
$k_0=\mathrm{id}$ on $\Pi$ and each face goes to itself; use the Alexander trick face by face rel
boundary, $k_t\times\mathrm{id}$ fixing $\lam(\Pi)$ pointwise. (4) Then $\psi\simeq\mathrm{id}$ on $N$
by an isotopy that is a product on $N$; extend it over the two complementary cap balls by coning, since
they contain no part of $K$. Now $\psi$ is the identity on $N$ and on the boundary of each cap, and the
Alexander trick in each cap finishes.
\end{proof}

\begin{proposition}\label{prop:surj-plus}
Every orientation-preserving element of $\mathrm{MCG}(S^3,K)$ is in the image of $\Phi$.
\end{proposition}

\begin{proof}
Let $\Phi_0$ be an orientation-preserving symmetry of $(S^3,\lam(\Pi))$. By Theorem~\ref{thm:MT},
$\Phi_0\simeq f_n\circ\dots\circ f_1$ with flypes $f_i\colon\lam(D_{i-1})\to\lam(D_i)$, $D_0=D_n=\Pi$.
Write $f_i=g_1\circ f'\circ g_2$.

\emph{Claim:} every $D_i$, and every diagram met in between, has underlying map isomorphic to $M$ or to
its mirror and is reduced, alternating, polyhedral and bigon-free. It holds for $D_0=\Pi$ by
Lemmas~\ref{lem:B1} and~\ref{lem:B2} and because the faces of $\Pi$ have degrees $3$, $4$ and $q\ge3$
(Section~\ref{sec:diagram}); flat maps keep it; and by Lemma~\ref{lem:A3} each $f'$ is
degenerate, its target has shadow isomorphic to $D$ ($S_A$ empty) or to $R(D)$, a mirror map ($S_B$
empty), and it is alternating by the Murasugi argument in Lemma~\ref{lem:F}. So Lemmas~\ref{lem:A3}
and~\ref{lem:F} apply at every step, each $f_i$ is isotopic to a map in $\cF$, and so is the composite:
$\Phi_0\simeq\varphi\in\cF$ with $\varphi(\lam(\Pi))=\lam(\Pi)$.

Near $S^2$, $\varphi=h_0\times\varepsilon$. Since $\varphi$ maps $\lam(\Pi)$ to itself, $h_0$ maps $\Pi$
to $\Pi$, inducing $\alpha\in\Aut(M)$, and $\varepsilon$ obeys the over/under rule for $\alpha$. By
Section~\ref{sec:count}, $|\Aut(M)|\le4q$, $\varepsilon$ is determined by $\alpha$, and $\Gamma$ realises
$4q$ distinct automorphisms; so some $g\in \Gamma$ induces $\alpha$ with the same $\varepsilon$. By
Lemma~\ref{lem:S} with $\varphi'=g$, $\Phi_0\simeq g$.

If $\Phi_0$ reverses the orientation of $K$, apply Theorem~\ref{thm:MT} with $D_2=\Pi$ reversed; the
map $M$ is the same. The maps of~\cite{MT93} respect the link orientation (``preserving orientations of
graphs and ambient spaces'', \cite[p.~117]{MT93}); their remark on reversibility
\cite[Corollary~1.1, pp.~116--117]{MT93} is this use.
\end{proof}

\subsection{Surjectivity, orientation-reversing symmetries}

\begin{proposition}\label{prop:surj-minus}
Every orientation-reversing element of $\mathrm{MCG}(S^3,K)$ is in the image of $\Phi$.
\end{proposition}

\begin{proof}
Let $\Phi_0$ reverse the orientation of $S^3$, and let $r=(\mathrm{id},-1)$, the reflection in $S^2$,
which lies in $\cFpm$ with $\det\cdot\varepsilon=-1$. Then $r\circ\Phi_0\colon\lam(\Pi)\to\lam(\Pi^*)$
preserves orientation, where $\Pi^*$ is $\Pi$ with all crossings switched (same map $M$).
Theorem~\ref{thm:MT} with $D_2=\Pi^*$ (reversed if $r\circ\Phi_0$ reverses the orientation of $K$, as in Proposition~\ref{prop:surj-plus}), the induction in the proof of Proposition~\ref{prop:surj-plus}
(with $D_n=\Pi^*$) and Lemma~\ref{lem:F} give $r\circ\Phi_0\simeq\varphi\in\cF$. So
$\Phi_0\simeq r^{-1}\circ\varphi$, which near $S^2$ is $h_0\times(-\varepsilon)$ with
$\operatorname{sign}(h_0)\cdot(-\varepsilon)=-1$ and which maps $\lam(\Pi)$ to itself. As before some $g\in \Gamma$
induces the same automorphism with the same sign $-\varepsilon$, and Lemma~\ref{lem:S} (for $\cFpm$)
gives $\Phi_0\simeq g$.
\end{proof}

\subsection{Pair and exterior}\label{subsec:pair}
By uniqueness of regular neighbourhoods and isotopy extension, a homeomorphism of the pair can be isotoped
to preserve $E$, and a pairwise isotopy between two such homeomorphisms can be replaced by one preserving
$E$ throughout. So restriction gives a well-defined map $\mathrm{MCG}(S^3,K)\to\mathrm{MCG}(E)=\Sym(S^3,K)$; composed with $\Phi$ it gives
$\Sym(\Pi)\to\Sym(S^3,K)$. The composite is injective, since
the proof of Proposition~\ref{prop:inj} works on $E$. It is surjective: a homeomorphism $h$ of $E$ sends
the meridian to $\pm$ the meridian, since otherwise Dehn filling would give $S^3$ by a non-trivial
surgery on the non-trivial knot $K$, which Gordon and Luecke~\cite{GL89} exclude; so $h$ extends to a
homeomorphism of the pair, which is isotopic to some $g\in \Gamma$ by Propositions~\ref{prop:surj-plus}
and~\ref{prop:surj-minus}, and $h$ is isotopic to $g|_E$ by the same remark. Hence both mapping class groups are
isomorphic to $\Sym(\Pi)$.

\section{Proof of Theorem~\ref{thm:main}}
\label{sec:proof}

By Propositions~\ref{prop:inj}, \ref{prop:surj-plus} and~\ref{prop:surj-minus},
$\Phi\colon\Sym(\Pi)\to\mathrm{MCG}(S^3,K)$ is an isomorphism, and by Section~\ref{subsec:pair} so is its composite with restriction to $E$,
$\Sym(\Pi)\to\Sym(S^3,K)$; by Lemmas~\ref{lem:lower}
and~\ref{lem:dihedral}, $\Sym(\Pi)\cong D_{2q}$ of order $4q$; hence $\Sym(S^3,K)\cong D_{2q}$, and by Section~\ref{subsec:conventions}
$\Sym(S^3,\Th(p,q))\cong D_{2q}$. Corollary~\ref{cor:chiral} follows by Remark~\ref{rem:chirality} and Section~\ref{subsec:conventions}. \qed

\begin{remark}
What the proof cites without re-proving: the Main Theorem of~\cite{MT93} with their definitions of flat
map and flype~\cite[pp.~115--116]{MT93}; Borel's theorem, via~\cite{CWY}; \cite[Theorems~2 and~3]{Mur87};
Gordon--Luecke~\cite{GL89}, only for passing from the pair to the exterior;
and standard facts: hyperbolic knot exteriors are aspherical with incompressible boundary and centreless
$\pi_1$, hyperbolic knots are prime, the peripheral subgroup is malnormal, the centre of a
non-trivial amalgam over $P$ lies in $P$, and the independence of the category noted in
Section~\ref{sec:reduction}. For general polyhedral diagrams the argument would need an
isometric realisation of the diagram automorphisms, e.g.\ via Koebe--Andreev--Thurston; for $\Th(p,q)$
the drawing of Section~\ref{subsec:drawing} is explicit. Schrijver's flyping theorem for well-connected
diagrams~\cite{Sch93} is the diagram-level statement matching Lemma~\ref{lem:B1}; it gives no statement
about maps and is not used.
\end{remark}

\section{Comparison with Sakuma--Weeks for \texorpdfstring{$p=3$}{p=3}}
\label{sec:SW}

\emph{What~\cite{SW95} prove.} $K(n)$ is the closure of the $3$-braid
$(\sigma_1\sigma_2^{-1})^n$~\cite[Example~I.2.2]{SW95}, so $K(q)=\Th(3,q)$. For $n\ge4$,
$\Isom(S^3-K(n))\cong D_{2n}$, of order $4n$; for $n=3$, the Borromean rings, the group is
$O\times\Z_2$~\cite[Proposition~I.2.5(2)]{SW95}. $\Th(3,q)$ is a knot if and only if $3\nmid q$, so for
knots their range $n\ge4$ is exactly the range of Theorem~\ref{thm:main} for $p=3$. (The survey treats
$\Th(3,q)$ under the assumption $\gcd(3,q)=1$~\cite[\S5]{DS} and states Theorem~M ``for any $q$'' in that
setting; \cite[Proposition~I.2.5]{SW95} covers $n\ge4$ there, and $q=2$, the figure-eight knot, lies
outside both their range and ours.)

\emph{How they prove it.} Following Thurston~\cite[6.8.11]{Th}, $S^3-K(n)$ is glued from two ideal
polyhedra $P(n)$ whose faces are regular in the Klein model, hyperbolically and euclideanly; this gives
the complete hyperbolic structure~\cite[pp.~399--401]{SW95}. Sakuma and Weeks introduce the construction
through the decomposition of an alternating link complement into two polyhedra read off the
diagram~\cite[p.~397, Proposition~I.2.1]{SW95}, crediting Menasco, Aitchison--Rubinstein and
others~\cite{Me83,AR1}. A centre criterion~\cite[Proposition~I.1.4]{SW95} shows that this is the
canonical (Epstein--Penner~\cite{EP}) decomposition. By Mostow rigidity the mapping class group of the
complement is the combinatorial automorphism group of the canonical decomposition~\cite[p.~393]{SW95};
that this group is the one in the proposition is ``easy to see''~\cite[p.~402]{SW95}.

\emph{How we prove it.} The route here is topological and diagrammatic: Theorem~\ref{thm:MT}, with Borel's theorem for injectivity, reduces
$\mathrm{MCG}(S^3,K)$ to $\Sym(\Pi)$ once Sections~\ref{sec:diagram} and~\ref{sec:polyhedral} have excluded bigons and non-trivial flypes,
Section~\ref{subsec:pair} (with Gordon--Luecke) passes to $\Sym(S^3,K)$, and the
$4q$ symmetries are counted in Section~\ref{sec:count}. No hyperbolic structure is constructed.

\emph{Where the two meet.} Both arguments end in a finite combinatorial count: in~\cite{SW95} the
automorphism group of the two polyhedra and their gluing, not written out; here the symmetry group of the
$4$-valent map with over/under data, written out in Section~\ref{sec:count}.

\emph{What each gives.} \cite{SW95} give the canonical decomposition itself, hence the hyperbolic
structure explicitly, and the isometry group also for the link $K(3)$. Their first step needs the
explicit regular structure, and by their own Remark~I.1.5 the canonical decompositions of almost all
hyperbolic knots up to $10$ crossings do not satisfy the assumption of their criterion. We do not know whether $\Th(p,q)$ for $p\ge4$ has a decomposition
to which their Proposition~I.1.4 applies; the argument here does not need one. It covers every
$p\ge3$ at once, and it rests instead on the Main Theorem of~\cite{MT93},
on Borel's theorem, on~\cite{Mur87} and, for the exterior, on~\cite{GL89}.

\emph{The relation they point to.} In \S III.1 Sakuma and Weeks conjecture that the arc of every crossing
of a reduced alternating diagram of a hyperbolic alternating link is isotopic to an edge of the canonical
decomposition, and write: ``We hope this approach sheds new light on the Tait conjecture (which was proved
by [MT] using combinatorial methods) from the point of view of hyperbolic
geometry''~\cite[p.~436]{SW95}. For Turk's head knots the argument here runs the other way, from the
theorem of Menasco and Thistlethwaite to the symmetry group that Sakuma and Weeks obtained geometrically
for $p=3$. For $p=3$ the two agree: Section~\ref{sec:count} gives $\rho$, $\sigma$ and the $\mu_a$,
together $D_{2q}$ of order $4q$.

\section{Computational checks}
\label{sec:checks}

The proof does not rest on computation; the following checks test the explicit formulas and the model.
\begin{itemize}
\item The formulas of Section~\ref{sec:count}: $\rho$, $\sigma$ and all $\mu_a$ are map automorphisms
with the stated $\varepsilon$; constant $a$ fails; $\tau$ has order $2q$; $\mu_0\tau\mu_0=\tau^{-1}$; the $4q$
elements have distinct vertex images. $217$ of $217$ cases ($3\le p\le15$, $3\le q\le30$, $\gcd=1$).
\item The hand model of the diagram, read in the cylinder chart, is isomorphic, as a map with over/under data, to spherogram's diagram
of the braid closure, for even $p$ to its mirror image: $60$ of $60$ cases ($3\le p\le10$, $3\le q\le15$, $\gcd=1$). For odd $p$
the diagram is isomorphic to its own mirror image (Remark~\ref{rem:chirality}), so there both descriptions hold.
\item Independently of the hand model, on spherogram's PD code: Tait graph simple and $3$-connected,
$|\Sym(\Pi)|=4q$, dihedral, in $263$ of $263$ cases ($3\le p\le16$, $q\ge3$, $\gcd=1$, at most $200$ crossings).
Calibration: $8_{18}$ gives $16$, $9_{40}$ gives $12$, $10_{123}$ gives $20$ (these are $\Th(3,4)$, $\Th(4,3)$ and
$\Th(3,5)$ themselves, so the values are a consistency check, not an independent calibration); the diagrams with flypes
($6_1$, $7_6$, $8_3$, $8_{14}$) are correctly not accepted, but all four already fail because they have bigons, so
$3$-connectivity is not tested separately on a bigon-free diagram with a flype.
\item The symmetric drawing of Section~\ref{subsec:drawing} is planar, has four rays at every vertex,
and $\rho$, $\sigma$, all $\mu_a$ are affine maps of it onto itself: $42$ of $42$ cases ($3\le p\le9$,
$3\le q\le12$, $\gcd=1$); for $p\ge4$, with offsets $\tfrac{i-1}{p-1}$ the maps $\mu_a$ are not of the form $\theta\mapsto a-\theta$, that is, not isometries of the drawing. The cyclic order of the
four rays: $64$ of $64$ cases ($3\le p\le11$, $3\le q\le13$, $\gcd=1$).
\item SnapPy's isometry groups (numerical) have order $4q$ in $95$ of $95$ cases ($3\le p\le11$, $q\ge3$, $\gcd=1$, at most $90$ crossings); exact certification
only for $\Th(3,4)$ and $\Th(4,3)$.
\end{itemize}

\subsection*{Acknowledgements}
We thank William Menasco for sending us the Annals paper~\cite{MT93} and the paper~\cite{Me83}, and O\u{g}uz \c{S}avk for
the reply to our question and for suggesting this write-up and the comparison with Sakuma and Weeks.

\subsection*{Tool use}
This work was prepared with substantial use of Claude Code (Anthropic), an agentic AI system based on
large language models, working together with the author.
\begin{itemize}
\item \emph{Mathematics.} Conjecture~R was among the open problems found in a literature search carried
out with Claude Code; the author chose it. The proof was worked out by Claude Code under the author's direction.
\item \emph{Writing and typesetting.} The text and the \LaTeX{} source were written by Claude Code, in
English; no translation tools were used. The paper contains no figures.
\item \emph{References.} Literature was searched with Claude Code; the sources were obtained by the author and by Claude Code, and
W.~Menasco sent us~\cite{MT93} and~\cite{Me83}. Quotations and page
references were checked against the original texts where these were available to us; bibliographic data
were checked against Crossref or against the sources themselves.
\item \emph{Code.} The checks of Section~\ref{sec:checks} were written and run by Claude Code in Python,
using SnapPy~3.3.2, spherogram~2.4.1 and networkx~3.6.1, and SageMath~10.9 for the exact certification.
The deposited scripts, at \url{https://doi.org/10.5281/zenodo.23234787} (the concept DOI of the deposit), are English versions of the working
scripts, prepared by Claude Code and run again before deposit with the same results; the deposit lists the
exact software versions.
\item \emph{Checking.} Drafts were read at least seventeen times by separate instances of Claude Code (in full, as the changes between
versions, or as targeted readings of single passages), without access to
the discussion in which the arguments were developed; their findings were checked and incorporated where
they held. Version~2 was read by a further Claude Code session on 24 October 2026. On 14 November 2026 it was read in full
by a Claude Code agent, by Claude (Anthropic) in a separate web-browser session and by DeepSeek, to which the author gave
the text; version~3 corrects the findings that held (listed in the README of the deposit). From 14 November 2026 on, the changes of version~3 were read in rounds, each by a fresh Claude Code agent, by Claude in a
web-browser session and by DeepSeek; the first round was given the changes from version~2 made by then, each
later round the changes made since the round before.
\end{itemize}
The author takes full responsibility for all contents.



\begin{thebibliography}{MT93}

\bibitem[AR1]{AR1} I.~R. Aitchison and J.~H. Rubinstein, \emph{Combinatorial cubings, cusps and the
dodecahedral knots}, in: Topology '90 (B.~Apanasov et al., eds.), de Gruyter, 1992, 17--26.

\bibitem[Bal61]{Bal61} M.~L. Balinski, \emph{On the graph structure of convex polyhedra in $n$-space},
Pacific J.\ Math.\ \textbf{11} (1961), 431--434.

\bibitem[Bo19]{Bo19} K.~Boyle, \emph{Odd order group actions on alternating knots},
\href{https://arxiv.org/abs/1906.04308}{arXiv:1906.04308v1} (2019).

\bibitem[CWY]{CWY} S.~Cappell, S.~Weinberger and M.~Yan, \emph{Closed aspherical manifolds with
center}, J.~Topol.\ \textbf{6} (2013), 1009--1018; arXiv:1108.2321. Page numbers refer to the arXiv version~v3.

\bibitem[DS]{DS} A.~Di Prisa and O.~\c{S}avk, \emph{Turk's head knots and links: a survey}, Expo.\
Math.\ \textbf{44} (2026), 125777; arXiv:2409.20106. Section numbers refer to the arXiv version~v3.

\bibitem[EP]{EP} D.~B.~A. Epstein and R.~C. Penner, \emph{Euclidean decompositions of noncompact
hyperbolic manifolds}, J.~Differential Geom.\ \textbf{27} (1988), 67--80.

\bibitem[GL89]{GL89} C.~McA. Gordon and J.~Luecke, \emph{Knots are determined by their complements},
J.~Amer.\ Math.\ Soc.\ \textbf{2} (1989), 371--415.

\bibitem[Me83]{Me83} W.~Menasco, \emph{Polyhedra representation of link complements}, Contemp.\ Math.\
\textbf{20} (1983), 305--325.

\bibitem[Men84]{Men84} W.~Menasco, \emph{Closed incompressible surfaces in alternating knot and link
complements}, Topology \textbf{23} (1984), 37--44.

\bibitem[MR90]{MR90} R.~Meyerhoff and D.~Ruberman, \emph{Cutting and pasting and the $\eta$-invariant},
Duke Math.\ J.\ \textbf{61} (1990), 747--761.

\bibitem[MT91]{MT91} W.~Menasco and M.~Thistlethwaite, \emph{The Tait flyping conjecture}, Bull.\ Amer.\
Math.\ Soc.\ \textbf{25} (1991), 403--412.

\bibitem[MT93]{MT93} W.~Menasco and M.~Thistlethwaite, \emph{The classification of alternating links},
Ann.\ of Math.\ \textbf{138} (1993), 113--171.

\bibitem[Mur87]{Mur87} K.~Murasugi, \emph{Jones polynomials and classical conjectures in knot theory},
Topology \textbf{26} (1987), 187--194.

\bibitem[Ow18]{Ow18} N.~Owad, \emph{Families of not perfectly straight knots},
\href{https://arxiv.org/abs/1804.04799}{arXiv:1804.04799v2} (2018).

\bibitem[QHW]{QHW} C.~V. Quach Hongler and C.~Weber, \emph{Link projections and flypes},
arXiv:0906.2059.

\bibitem[Sau26]{Sau26} S.~Saurabh, \emph{Equivariant rational sliceness and Klein amphichirality of odd-stranded
Turk's head knots}, \href{https://arxiv.org/abs/2607.14638}{arXiv:2607.14638v1} (2026).

\bibitem[Sch93]{Sch93} A.~Schrijver, \emph{Tait's flyping conjecture for well-connected links},
J.~Combin.\ Theory Ser.~B \textbf{58} (1993), 65--146.

\bibitem[SGJ]{SGJ} A.~Shimizu, A.~Gill and S.~Joshi, \emph{A note on the unknotting number and the region unknotting
number of weaving knots}, \href{https://arxiv.org/abs/2408.14938}{arXiv:2408.14938v3} (2025).

\bibitem[SW95]{SW95} M.~Sakuma and J.~Weeks, \emph{Examples of canonical decompositions of hyperbolic
link complements}, Japan.\ J.\ Math.\ (N.S.) \textbf{21} (1995), 393--439.

\bibitem[Th]{Th} W.~Thurston, \emph{The geometry and topology of three-manifolds}, lecture notes;
electronic version~1.1, March 2002, \url{https://library.slmath.org/books/gt3m/} (Example~6.8.11 is on p.~154 there).

\end{thebibliography}
\end{document}